\documentclass[11pt]{amsart}

\usepackage{amsfonts, amssymb, amscd}

\usepackage{verbatim}
\usepackage{amssymb}
\usepackage{mathrsfs}
\usepackage{graphicx}
\usepackage{cite}
\usepackage[all]{xy}
\usepackage{tikz}
\usepackage[toc,page]{appendix}
\usepackage{tikz-cd}

\usepackage{graphicx}
\usepackage{float}

\newtheorem{theorem}{Theorem}[section]
\newtheorem{lemma}[theorem]{Lemma}
\newtheorem{proposition}[theorem]{Proposition}
\newtheorem{definition}[theorem]{Definition}
\newtheorem{example}[theorem]{Example}
\newtheorem{corollary}[theorem]{Corollary}
\newtheorem{remark}[theorem]{Remark}
\newtheorem{conjecture}[theorem]{Conjecture}
\numberwithin{equation}{section}
\begin{document}
\title{On $\mathrm{H}-$trivial line bundles on toric DM stacks of dimension two}

\begin{abstract}
We study line bundles on toric DM stacks $\mathbb{P}_{\mathbf{\Sigma}}$ of dimension two. We give a combinatorial criterion of when infinitely many line bundles on $\mathbb{P}_{\mathbf{\Sigma}}$ have trivial cohomology. We further discuss the structure of the set of such line bundles.
\end{abstract}

\author{Chengxi Wang}
\address{Department of Mathematics\\
Rutgers University\\
Piscataway, NJ 08854} \email{cw674@math.rutgers.edu}

\maketitle

\tableofcontents
\section{Introduction}\label{section1}
Derived categories of coherent sheaves on toric varieties and DM stacks provide examples of combinatorially defined triangulated categories. A lot of work has been done over the years aimed at finding exceptional objects and collections in these categories in \cite{Ka}. In particular, line bundles provide examples of exceptional objects. The description of line bundles on the DM stacks is analogous to the description of the Picard group that was given in \cite{D,F}.  In 1997 Alastair King conjectured in \cite{Ki} the following.
\begin{conjecture}\label{nFs}
Every smooth toric variety possesses a full strong exceptional collection of line bundles.
\end{conjecture}
The conjecture was shown to be false in \cite{HP1}, but multiple negative and positive results in this general direction have been obtained in \cite{HP2,M,E,K,I}.
In \cite{E}
Efimov also disproves weaker version of King's conjecture proposed by Costa and Mir\'{o}-Roig in \cite{CM},
and Craw through constructing infinitely many examples of toric Fano varities with Picard number three, which do not have full exceptional collections of line bundles. Conjecture \ref{nFs} turned out to be true for smooth toric Fano DM stacks of Picard number at most two and of any Picard number in dimension two \cite{K}. Further, a full strong exceptional collection of line bundles is constructed on any smooth toric nef-Fano Del Pezzo stack in
\cite{I}.

Since the requirement $\mathrm{Ext}^i(\mathcal{L}_1, \mathcal{L}_2)= 0$ in the definition of exceptional collections of line bundles translates into $\mathrm{H}^i(\mathcal{L}_2\otimes\mathcal{L}_1^{-1})=0$,
it is natural to study line bundles with trivial cohomology spaces. In our paper, we are interested in such $\mathrm{H}-$trivial line bundles.

It is meaningful to ask when there are infinitely many $\mathrm{H}-$trivial line bundles. One way to get infinitely many $\mathrm{H}-$trivial line bundles is to have a fibration $\pi$ from $\mathbb{P}_{\mathbf{\Sigma}}$ to a certain base and a line bundle $\mathcal{L}$ on $\mathbb{P}_{\mathbf{\Sigma}}$ such that the higher direct images $\mathbf{R}^i \pi_*(\mathcal{L})=0$ for all $i\geq 0$. The reason is that by the Leray spectral sequence, we get $H^i(\mathbb{P}_{\mathbf{\Sigma}},\mathcal{L})=0$ for $i\geq 0$. Importantly, for any line bundle $\mathcal{F}$ from the base, the higher direct images $\mathbf{R}^i \pi_*(\mathcal{L}\otimes \pi^* \mathcal{F})=0$ for all $i\geq 0$, which implies $\mathcal{L}\otimes \pi^* \mathcal{F}$ is $\mathrm{H}-$trivial. Thus we may obtain infinitely many $\mathrm{H}-$trivial line bundles on $\mathbb{P}_{\mathbf{\Sigma}}$.

In this paper, we find a combinatorial criterion for when there exist infinitely many $\mathrm{H}-$trivial line bundles on $\mathbb{P}_{\mathbf{\Sigma}}$ for smooth toric varieties and DM stacks in dimension two.

\smallskip

{\bf Theorem \ref{iff}.}
Let $\mathbb{P}_{\mathbf{\Sigma}}$ be a proper smooth dimension two toric DM stacks associated to a complete stacky fan $\mathbf{\Sigma}=(\Sigma, \{v_i\}_{i=1}^{n})$. Then
there are infinitely many $\mathrm{H}-$trivial line bundles on $\mathbb{P}_{\mathbf{\Sigma}}$ if and only if there exists $\{i, j\}\subset\{1,2,\cdots,n\}$ such that $v_i$ and $v_j$ are collinear.

\smallskip

Moreover, when there are infinitely many $\mathrm{H}-$trivial line bundles, we display the tubes $+$ ball description of $\mathrm{H}-$trivial line bundles in Theorem \ref{BT1}. That is to say that the set of $\mathrm{H}-$trivial line bundles can be depicted in the form "finite set $+$ finite set of lines ", which is revealed by Proposition \ref{line} and Theorem \ref{BT1}. Further, we generalize our results to broader classes $\Lambda_m$ of line bundles with $\mathrm{dim}\, \mathrm{H}^0+ \mathrm{dim}\, \mathrm{H}^1+ \mathrm{dim}\, \mathrm{H}^2<m$ for some positive integer $m$.

The paper is organized as follows.
In Section \ref{stack}, we review smooth toric DM stacks and their Picard groups. Also, we introduce the cohomology of line bundles on the stacks, the definition of forbidden cones and forbidden sets, and state Theorem \ref{iff}. Section \ref{H} focuses on the proof of Theorem \ref{iff}. We first prove the 'if' direction in Proposition \ref{->inf}. As for the 'only if' direction, the key step is to show that any non-zero element $\mathcal{L}\in \mathrm{Pic}_{\mathbb{R}}(\mathbb{P}_{\mathbf{\Sigma}})$ is contained in the interior of some forbidden cone shifted to the origin (see Lemma \ref{interior}). Section \ref{description} further reveals the tubes $+$ ball structure of the set of $\mathrm{H}-$trivial line bundles. Section \ref{gene} contains generalization of our results to $\Lambda_m$. Section \ref{comments} proposes a conjecture in dimension $3$ case.
In Appendix \ref{appendix}, we collect several facts about cones in a finitely generated abelian group which are used in the proof of the results in former sections.

\smallskip

\noindent{\it Acknowledgements.} The author thanks Lev Borisov for multiple useful comments.

\section{Line bundles on toric DM stacks and their cohomology}\label{stack}
In this section, we introduce toric DM stacks $\mathbb{P}_{\mathbf{\Sigma}}$ and their Picard groups $\mathrm{Pic} (\mathbb{P}_{\mathbf{\Sigma}})$, and describe the cohomology of line bundles on $\mathbb{P}_{\mathbf{\Sigma}}$.
We formulate our main result which is a criterion for dimension two proper toric DM stacks to have infinitely  many line bundles with trivial cohomology.

In order to refrain from technicalities of the derived Gale duality of \cite{BCS}, we consider a lattice
$N$ which is a free abelian group of finite rank. Let $\Sigma$ be a complete simplicial fan in $N$. We choose a lattice point $v$ in each of the one-dimensional cones of $\Sigma$. If $\Sigma$ has $n$ one-dimensional cones, we get a complete stacky fan $\mathbf{\Sigma}=(\Sigma, \{v_i\}_{i=1}^{n})$, see \cite{BCS}.

The toric DM stack $\mathbb{P}_{\mathbf{\Sigma}}$ associated to this stacky fan $\mathbf{\Sigma}$ is constructed as follows, see \cite{BCS,C}. We have a natural map with finite cokernel
$$\sigma: \mathbb{Z}^n \rightarrow N$$ given by $(\mu_1,\ldots,\mu_n)\mapsto \sum_{i=1}^{n}\mu_i v_i$. Taking the dual injective map $$\sigma^*: N^* \rightarrow \mathbb{Z}^n ,$$ we get the cokernel of $\sigma^*$ which we denote by $\mathrm{Gale}(N)$. Define the algebraic Group $G$ by $$G:=\mathrm{Hom}(\mathrm{Gale}(N),\mathbb{C}^*).$$ Now we have an injection induced by $\sigma^*$ $$G\subseteq (\mathbb{C}^*)^n$$ and $G=\{(\zeta_1,\ldots,\zeta_n)\in(\mathbb{C}^*)^n| \prod_{i=1}^{n}\zeta_i^{w\cdot v_i}=1 \text{ for all }w\in N^*\}$, where $w\cdot v_i$ is natural pairing.
Define the open set $U$ in $(\mathbb{C}^*)^n$ as follows. A point $(x_1,\ldots,x_n)\in(\mathbb{C}^*)^n$ lies in $U$ if and only if there exists
a cone in $\Sigma$ which contains all $\{v_i|x_i=0\}$. We have a natural action of $G$ on $U$ via inclusion $G\subseteq (\mathbb{C}^*)^n$. Group $G$ acts with finite isotropy subgroups and we define by $\mathbb{P}_{\mathbf{\Sigma}}$ the DM stack $[U/G]$, see \cite{BCS}.

By \cite{V}, we know the category of coherent sheaves on $\mathbb{P}_{\mathbf{\Sigma}}$ is equivalent to the category of $G-$equivariant sheaves on $U$. Specifically, the line bundles on $\mathbb{P}_{\mathbf{\Sigma}}$ are described explicitly as follows.

\begin{definition}\label{lb}
For any element $(r_1,\ldots,r_n)\in \mathbb{Z}^n$, we have the trivial line bundle $\mathbb{C}\times U\rightarrow U$ with the $G-$linearization $G \times \mathbb{C}\times U\rightarrow \mathbb{C}\times U$ denoted by $$((\zeta_1,\ldots,\zeta_n), t, (x_1,\ldots,x_n))\mapsto(t\prod_{i=1}^{n}\zeta_i^{r_i},(\zeta_1x_1,\ldots,\zeta_nx_n)).$$By \cite{V}, this gives a line bundle on $\mathbb{P}_{\mathbf{\Sigma}}$ and we denote the corresponding invertible sheaf in Picard group of $\mathbb{P}_{\mathbf{\Sigma}}$ by $\mathcal{O}(\Sigma_{i=1}^{n}r_iE_i)$.
\end{definition}

Then we have the following proposition to describe the Picard group of $\mathbb{P}_{\mathbf{\Sigma}}$ which we denote by $\mathrm{Pic}(\mathbb{P}_{\mathbf{\Sigma}})$.
\begin{proposition}\label{pic}
We obtain all line bundles on $\mathbb{P}_{\mathbf{\Sigma}}$ by the construction of Definition \ref{lb}. The Picard group of $\mathbb{P}_{\mathbf{\Sigma}}$ is isomorphic to the quotient of $\mathbb{Z}^n$ with basis $\{E_i\}_{i=1}^{n}$ by the subgroup of elements of the form $\sum_{i=1}^{n}(w_i\cdot v_i)E_i$ for all $w\in N^*$.
\end{proposition}
\begin{proof}
See \cite{K}.
\end{proof}

Now we remind the reader how to calculate the cohomology of a line bundle $\mathcal{L}$ on $\mathbb{P}_{\mathbf{\Sigma}}$. For each $\mathbf{r}=(r_i)_{i=1}^{n}\in \mathbb{Z}^n$, we define $Supp(\mathbf{r})$ to be the simplicial complex on $n$ vertices $\{1,\ldots,n\}$ as follows
\begin{equation*}
\begin{split}
Supp(\mathbf{r})=&\{J \subseteq \{1,\ldots,n\}| r_i\geq 0 \text{ for all }i\in J\\
&\text{ and there exists a cone of $\Sigma$ containing all }v_i, i\in J \}.
\end{split}
\end{equation*}

The following proposition gives a description of the cohomology of a linear bundle $\mathcal{L}$ on $\mathbb{P}_{\mathbf{\Sigma}}$.
\begin{proposition}\label{Coho}
\cite{K} Let $\mathcal{L} \in \mathrm{Pic}(\mathbb{P}_{\mathbf{\Sigma}})$. Then $$\mathrm{H}^j(\mathbb{P}_{\mathbf{\Sigma}},\mathcal{L})=\bigoplus \mathrm{H}^{red}_{rk N-j-1}(Supp(\mathbf{r})),$$ where the sum is over all $\mathbf{r}=(r_i)_{i=1}^{n}\in \mathbb{Z}^n$ such that $\mathcal{O}(\sum_{i=1}^{n}r_iE_i)\cong \mathcal{L}$.
\end{proposition}
\begin{proof}
See \cite{K}.
\end{proof}
\begin{remark}\label{extreme}
We have $\mathrm{H}^0(\mathcal{L})\neq0$ if and only if there exists $\mathbf{r}\in \mathbb{Z}_{\geq 0}^n$ such that $\mathcal{O}(\sum_{i=1}^{n}r_iE_i)\cong \mathcal{L}$. Another extreme case is that $\mathrm{H}^{rk(N)}(\mathcal{L})$ only appears when the simplicial complex $Supp(\mathbf{r})=\{\emptyset\}$, i.e. when $\mathcal{O}(\sum_{i=1}^{n}r_iE_i)\cong \mathcal{L}$ with all $r_i\leq -1$.
\end{remark}

\begin{remark}\label{coho2}
Let $\mathcal{L}\cong \mathcal{O}(\sum_{i=1}^{n}a_iE_i) $ be a line bundle in $\mathrm{Pic} (\mathbb{P}_{\mathbf{\Sigma}})$. Assume there is another expression $\mathcal{L} \cong \mathcal{O}(\sum_{i=1}^{n}r_iE_i)$. Then by Proposition \ref{pic}, there exists an element $f\in N^*$ such that $r_i=a_i+f(v_i)$ for $i=1,\ldots,n$, where $f(v_i)=(f.v_i)$. Thus the cohomology of $\mathcal{L}$ can also be written as following:
$$\mathrm{H}^j(\mathbb{P}_{\mathbf{\Sigma}},\mathcal{L})=\bigoplus_{f\in N^*}\mathrm{H}^{red}_{rk N-j-1}(Supp(\mathbf{r}_f)),$$where $\mathbf{r}_f=(a_i+f(v_i))_{i=1}^{n}$.
\end{remark}

In this paper, our primary objects of interest are $\mathrm{H}-$trivial line bundles which we define below.

\begin{definition}
Let $\mathcal{L} $ be a line bundle in $\mathrm{Pic} (\mathbb{P}_{\mathbf{\Sigma}})$. We say that $\mathcal{L}$ is $\mathrm{H}-$trivial iff $\mathrm{H}^j(\mathbb{P}_{\mathbf{\Sigma}},\mathcal{L})=0$ for all $j\geq 0$.
\end{definition}

A combinatorial criterion for $\mathrm{H}-$triviality is given in terms of forbidden sets introduced below, see \cite{K}.

\begin{definition}
For every subset $I \subseteq \{1,\ldots,n\}$, we denote $C_I$ to be the simiplicial complex $Supp(\mathbf{r})$ where $r_i=-1$ for $i \notin I$ and $r_i=0$ for $i\in I$. Let $\Delta=\{I \subseteq \{1,\ldots,n\}| C_I \text{ has nontrivial reduced homology }\}$. By Remark \ref{extreme}, $\Delta$ contains $\{1,\ldots,n\}$ and $\emptyset$. For each $I\in \Delta$, the forbidden set associated to $I$ is defined by $$FS_{I}:=\{\mathcal{O}(\sum_{i\notin I}(-1-r_i)E_i+\sum_{i\in I}r_iE_i)|r_i\in \mathbb{Z}_{\geq 0} \text{ for all } i\}.$$
\end{definition}

\begin{proposition}\label{Htrivial}
Let $\mathcal{L}$ be a line bundle on $\mathbb{P}_{\mathbf{\Sigma}}$. Then $\mathcal{L}$ is $\mathrm{H}-$trivial if and only if $\mathcal{L}$ does not lie in $FS_I$ for any $I\in \Delta$.
\end{proposition}
\begin{proof}
This follows immediately from Proposition \ref{Coho}.
\end{proof}

Let $\mathrm{Pic}_{\mathbb{R}}(\mathbb{P}_{\mathbf{\Sigma}})=\mathrm{Pic}(\mathbb{P}_{\mathbf{\Sigma}})\otimes \mathbb{R}$ which can be regarded as a quotient of $\mathbb{R}^n$ with basis elements given by $E_i$. We know $\mathrm{Pic}_{\mathbb{R}}(\mathbb{P}_{\mathbf{\Sigma}})$ is a vector space with dimension equal to the rank of $\mathrm{Pic}(\mathbb{P}_{\mathbf{\Sigma}})$.
\begin{definition}\label{ZI}
For each $I\in \Delta$, we define the forbidden point by $$q_I=-\sum_{i\notin I}E_i \in \mathrm{Pic}_{\mathbb{R}}(\mathbb{P}_{\mathbf{\Sigma}}).$$
Define a cone associated to $I$ with vertex at the origin to be $$Z_I=\sum_{i\in I}\mathbb{R}_{\geq 0}E_i-\sum_{i\notin I}\mathbb{R}_{\geq 0} E_i.$$
Define the forbidden cone $FC_I\subseteq \mathrm{Pic}_{\mathbb{R}}(\mathbb{P}_{\mathbf{\Sigma}}) $ by $$FC_I=q_I+Z_I.$$
\end{definition}
\begin{remark}
By definition, we have $FS_I\subseteq FC_I$ for any $I\in \Delta$.
\end{remark}
In this paper, we are mainly concerned with dimension two, i.e. $N=\mathbb{Z}^2$. We have a complete simplicial fan $\Sigma$ in $\mathbb{Z}^2$ with $n$ one-dimensional cones and $n$ lattice points $\{v_i\}_{i=1}^n$ chosen in each of the one-dimensional cones of $\Sigma$, see Figure \ref{fig:5}. The maximal cones of $\Sigma$ are $\mathbb{R}_{\geq 0}v_1+\mathbb{R}_{\geq 0}v_2, \mathbb{R}_{\geq 0}v_2+\mathbb{R}_{\geq 0}v_3,\ldots,\mathbb{R}_{\geq 0}v_n+\mathbb{R}_{\geq 0}v_1$. In dimension $2$ case, we describe $\Delta=\{\emptyset, \{1,\ldots,n\}\}\cup \{I\subset \{1,\ldots,n\}| C_I\text{ is disconnected}\}$. For example, we have $\{1,3\}\in \Delta$ if $n>3$, $\{n,2,3\}\in \Delta$ if $n>4$, but $\{1,2\}\notin \Delta$, $\{n,1,2\}\notin \Delta$ for all $n>2$.
\begin{figure}[H]
  \includegraphics[width=0.3\textwidth]{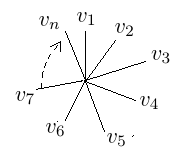}
  \caption{}
  \label{fig:5}
\end{figure}
Our first key result is a criterion for toric DM stacks to have infinitely many $\mathrm{H}-$trivial line bundles whose proof is give in the next section.
This key result is:
\begin{theorem}\label{iff}
Let $\mathbb{P}_{\mathbf{\Sigma}}$ be a proper smooth dimension two toric DM stacks associated to a complete stacky fan $\mathbf{\Sigma}=(\Sigma, \{v_i\}_{i=1}^{n})$. Then
there are infinitely many $\mathrm{H}-$trivial line bundles on $\mathbb{P}_{\mathbf{\Sigma}}$ if and only if there exists $\{i, j\}\subset\{1,2,\cdots,n\}$ such that $v_i$ and $v_j$ are collinear.
\end{theorem}
\begin{remark}
To illustrate the result, observe that $\mathbb{P}^2$ only has two $\mathrm{H}-$trivial line bundles $\mathcal{O}(-1)$ and $\mathcal{O}(-2)$, but $\mathbb{P}^1\times \mathbb{P}^1$ has infinitely many $\mathrm{H}-$trivial line bundles $\mathcal{O}(a,b)$, where $a=-1$ or $b=-1$. $(1)$ of Figure \ref{fig:8} is the fan corresponding to $\mathbb{P}^2$ and $(2)$ of Figure \ref{fig:8} is the fan corresponding to $\mathbb{P}^1\times \mathbb{P}^1$.
\end{remark}
\begin{figure}[H]
  \includegraphics[width=0.5\textwidth]{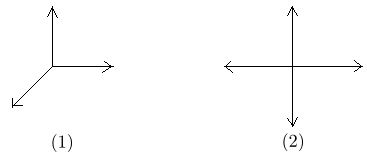}
  \caption{}
  \label{fig:8}
\end{figure}
\section{Proof of the main result}\label{H}
In this section, we give the proof of Theorem \ref{iff}.
We first prove the "if" direction of Theorem \ref{iff}.

Let $v_1,\ldots,v_t$ be $t$ vectors in $\mathbb{R}^2$ which are ordered clockwise. Let $f$ be a piecewise linear function on $\mathbb{R}^2$ which is linear in each cone spanned by $v_i$ and $v_{i+1}$. We regard $\geq0$ and $<0$ as different signs. We count the number of pairs of vectors $\{v_i,v_{i+1}\}\subset \{v_1,\ldots,v_t\}$ such that $f(v_i)$ and $f(v_{i+1})$ have different signs. We call it the number of sign changes of $f$ among $\{v_1,\ldots,v_t\}$.

\begin{lemma}\label{signch}
Let $v_1,\ldots,v_t$ be $t$ vectors in $\mathbb{R}^2$ which are ordered clockwise. Assume that $f$ is a linear function, angle $\theta$ between $v_1$ and $v_t$ is $\pi$ and $f(v_1)$ and $f(v_t)$ have different signs. Then the
the number of sign changes of $f$ among $\{v_1,\ldots,v_t\}$ is exactly one.
\end{lemma}
\begin{figure}[H]
  \includegraphics[width=0.23\textwidth]{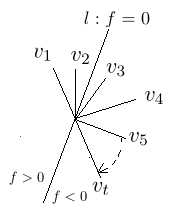}
  \caption{}
  \label{fig:7}
\end{figure}
\begin{proof}
Since $\theta= \pi$, we have that $v_1$ and $v_t$ are collinear. Also since $f(v_1)$ and $f(v_t)$ have different signs, the line $l=\{v\in\mathbb{R}^2|f(v)=0\}$ passes through the interior of the angle $\theta$. We know $f$ takes positive values on the vectors at one side of $l$, negative values on the vectors at another side of $l$, zero on the vectors on the line $l$. Thus there is exactly one sign change of $f$ among $\{v_1,\ldots,v_t\}$. See Figure \ref{fig:7}.
\end{proof}
\begin{proposition}\label{->inf}
In the assumption of Theorem \ref{iff}, if there is $\{i, j\}\subset\{1,2,\cdots,n\}$ such that $v_i$ and $v_j$ are collinear, $\mathbb{P}_{\Sigma}$ has infinitely many $\mathrm{H}-$trivial line bundles.
\end{proposition}
\begin{proof}
\begin{figure}[H]
  \includegraphics[width=0.9\textwidth]{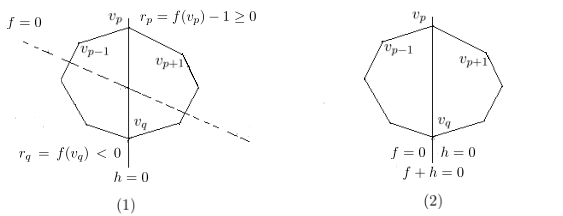}
  \caption{}
  \label{fig:2}
\end{figure}
We assume that $v_p$ and $v_q$ are collinear, and $v_1, v_2,\ldots, v_n$ are ordered clockwise. We take a linear function $h\in N^*$ such that $h(v_p)=h(v_q)=0$ and consider an element $\mathcal{L}=\mathcal{O}(\sum_{i\in I_{+}}h(v_i)E_i-E_p)$ in Picard group $\mathrm{Pic}(\mathbb{P}_{\mathbf{\Sigma}})$, where $I_{+}=\{i|h(v_i)>0\}$. Also let $I_{-}=\{i|h(v_i)<0\}$. By Remark \ref{coho2},
we have $$\mathrm{H}^j(\mathbb{P}_{\mathbf{\Sigma}},\mathcal{L})=\bigoplus_{f\in N^*}\mathrm{H}^{red}_{rk N-j-1}( Supp(\mathbf{r}_f)),$$where $\mathbf{r}_f=(a_i+f(v_i))_{i=1}^{n}$. In order to show $\mathrm{H}^*(\mathbb{P}_{\mathbf{\Sigma}},\mathcal{L})=0$, it is sufficient to show there are exactly two sign changes in $\mathbf{r}_f$ for each $f\in N^*$. Now we arbitrarily take an element $f\in N^*$. Let $\sum_{i}r_iE_i=\sum_{i\in I_{+}}h(v_i)E_i-E_p+\sum_{i=1}^{n}f(v_i)E_i$. We know that $r_i=(f+h)(v_i)$ for $i\in I_{+}\cup \{q\}$, $r_i=f(v_i)$ for $i\in I_{-}\cup \{q\}$ and $r_p=-1+f(v_p)$. We consider the following two cases: $r_p\neq -1$ and $r_p= -1$.

First, in the case that $r_p\neq-1$, we have $f(v_p)\neq0$ and therefore $f(v_q)\neq0$. They have different signs since $f$ is a linear function and $v_p$ and $v_q$ are collinear.  If $f(v_p)>0$ and $f(v_q)<0$, we have $r_q=f(v_q)<0$ and $r_p=f(v_p)-1\geq0$ since $f(v_p)$ is integer. Thus the sign of $f(v_i)$ is the same as the sign of $r_i$ for all $i\in I_{-}\cup\{p,q\}$. By Lemma \ref{signch}, the number of sign changes of $f$ among $\{v_i|i\in I_{-}\cup\{p,q\}\}$ is exactly one. Thus there is exactly one sign change in $\{r_i|i\in I_{-}\cup\{p,q\}\}$. See $(1)$ of Figure \ref{fig:2}. If $f(v_p)<0$ and $f(v_q)>0$, we have $r_q=f(v_q)>0$ and $r_p=f(v_p)-1<0$. Analogously, there is exactly one sign change in $\{r_i|i\in I_{-}\cup\{p,q\}\}$.
\begin{figure}[H]
  \includegraphics[width=1.1\textwidth]{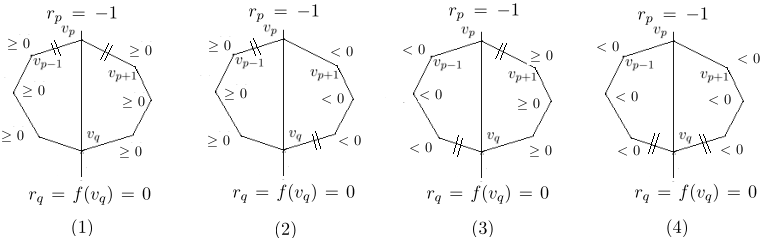}
  \caption{}
  \label{fig:6}
\end{figure}
 When $f(v_p)\neq0$ and $f(v_q)\neq0$, we also have $(f+h)(v_p)\neq0$ and $(f+h)(v_q)\neq0$. They have different signs since $f+h$ is a linear function and $v_p$ and $v_q$ are collinear. If $(f+h)(v_p)>0$ and $(f+h)(v_q)<0$, we have $f(v_p)>0$ and $f(v_q)<0$ since $h(v_p)=h(v_q)=0$. Then we have $r_q=f(v_q)<0$ and $r_p=f(v_p)-1\geq0$. If $(f+h)(v_p)<0$ and $(f+h)(v_q)>0$, we have $f(v_p)<0$ and $f(v_q)>0$ since $h(v_p)=h(v_q)=0$. Then we have $r_q=f(v_q)>0$ and $r_p=f(v_p)-1<0$. Thus the sign of $(f+h)(v_i)$ is the same as the sign of $r_i$ for all $i\in I_{+}\cup\{p,q\}$. By Lemma \ref{signch}, the number of sign changes of $f+h$ among $\{v_i|i\in I_{+}\cup\{p,q\}\}$ is exactly one. Thus there is exactly one sign change in $\{r_i|i\in I_{+}\cup\{p,q\}\}$. Thus there are exactly two sign changes in $(r_i)_{i=1}^{n}$.

Second, in the case that $r_p=-1$, we have $r_q=f(v_q)=f(v_p)=0$. Thus $(f+h)(v_p)=(f+h)(v_q)=0$. See $(2)$ of Figure \ref{fig:2}. Then we have $r_i=f(v_i)\geq 0$ for all $i\in I_{-}$ or $r_i=f(v_i)< 0$ for all $i\in I_{-}$. Also we have $r_i=(f+h)(v_i)\geq 0$ for all $i\in I_{+}$ or $r_i=(f+h)(v_i)< 0$ for all $i\in I_{+}$. There are four cases in total. In each case, there are exactly one sign changes in $\{r_i|i\in I_{+}\cup\{p,q\}\}$ and one sign changes in $\{r_i|i\in I_{-}\cup\{p,q\}\}$. Thus there are exactly two sign changes in $(r_i)_{i=1}^{n}$. See $(1), (2), (3), (4)$ of Figure \ref{fig:6}.

Thus we have infinitely many $\mathrm{H}-$trivial line bundles.
\end{proof}

Our next goal is to show that if there are infinitely many $\mathrm{H}-$trivial line bundles on the dimension two proper toric DM stack $\mathbb{P}_{\mathbf{\Sigma}}$, then there exists $\{i, j\}\subset\{1,2,\cdots,n\}$ such that $v_i$ and $v_j$ are collinear. First, we prove the following lemma.
\begin{lemma}\label{g}
For any non-zero element $\mathcal{L}\in \mathrm{Pic}_{\mathbb{R}}(\mathbb{P}_{\mathbf{\Sigma}})$ and any $j\in \{1,2,\ldots,n\}$, there exists a $\mathbf{\Sigma}-$piece-wise linear function $g$ on $\mathbb{R}^2$ such that $\mathcal{L}=\sum_{i=1}^{n}g(v_i)E_i$ in $\mathrm{Pic}_{\mathbb{R}}(\mathbb{P}_{\mathbf{\Sigma}})$ and $g(v_j)=g(v_{j+1})=0$.
\end{lemma}
\begin{proof}
Assume $\mathcal{L}=\sum_{i}a_iE_i$, where $a_i\in \mathbb{R}$, we have a $\mathbf{\Sigma}-$piece-wise linear function $h$ on $\mathbb{R}^2$ such that $h(v_i)=a_i$. Also, we take a linear function $f$ on $\mathbb{R}^2$ such that $f(v_j)=a_j$ and $f(v_{j+1})=a_{j+1}$ and let $g=h-f$. Then we have $g(v_j)=0$, $g(v_{j+1})=0$ and $\sum_{i}a_iE_i=\sum_{i}g(v_i)E_i$ in $\mathrm{Pic}_{\mathbb{R}}(\mathbb{P}_{\mathbf{\Sigma}})$.
\end{proof}
\begin{lemma}\label{interior}
Assume that $v_i$ and $v_j$ are not collinear for any $\{i, j\}\subset\{1,2,\cdots,n\}$. Then any non-zero element $\mathcal{L}\in \mathrm{Pic}_{\mathbb{R}}(\mathbb{P}_{\mathbf{\Sigma}})$ is contained in the interior of $Z_I$ from Definition \ref{ZI} for some $I\in \Delta$.
\end{lemma}
\begin{proof}
 By Lemma \ref{g}, we assume $\mathcal{L}=\sum_{i}g(v_i)E_i$, where $g$ is a $\mathbf{\Sigma-}$piece-wise linear function on $\mathbb{R}^2$ such that $g(v_1)=g(v_2)=0$. Let $I_0=\{i|g(v_i)=0\}$. Since $\mathcal{L}$ is nonzero, we have $I_0\neq \{1,2,\ldots,n\}$. Let $J$ be the connected component of $I_0$ containing $v_1$ and $v_2$. We write the points in $J$ in clockwise order as $v_k,\ldots,v_1,v_2,\ldots,v_l$. Then we consider the following cases.

 First, we consider the case when $g(v_{k-1})$ and $g(v_{l+1})$ have different signs. Without loss of generality, we assume $g(v_{k-1})>0$ and $g(v_{l+1})<0$. We can take a linear function $f'$ satisfying that $f'(v_2)>0$, $f'(v_1)<0$, $0\neq f'(v_i)$ for all $i\in I_0$ and $|f'(v_i)|<|g(v_i)|$ for all $i\in \{1,2,\ldots,n\}\backslash I_0 $. Let $g'=g+f'$. Then we have that $g'(v_{k-1})>0, g'(v_1)<0, g'(v_2)>0, g'(v_{l+1})<0$ and $g'(v_i)\neq0$ for $i=1,\ldots,n$. Thus we have $I_{>0}=\{i|g'(v_i)>0\}$ is not connected. So $I_{>0} \in \Delta$. Since $0\neq f'(v_i)$ for all $i\in I_0$, we have $g'(v_i)=(g+f)(v_i)\neq 0$ for all $i\in I_0$. Since $|f'(v_i)|<|g(v_i)|$ for all $i\in \{1,2,\ldots,n\}\backslash I_0 $, we have $g'(v_i)=(g+f)(v_i)\neq 0$ for $i\in \{1,2,\ldots,n\}\backslash I_0$. So we have $g'(v_i)\neq0$ for all $i$. Thus $$\mathcal{L}=\sum_{i\in I_{>0}}\alpha_iE_i+\sum_{i\notin I_{>0}}\alpha_iE_i$$ in $\mathrm{Pic}_{\mathbb{R}}(\mathbb{P}_{\mathbf{\Sigma}})$, where $\alpha_i=g'(v_i)\neq 0$ for all $i$ and $I_{>0}=\{i|g'(v)>0\}$. This implies that $\mathcal{L}$ is in the interior of $ C_{I_{>0}}$. See $(1)$ of Figure \ref{fig:1}.
\begin{figure}[H]
  \includegraphics[width=0.9\textwidth]{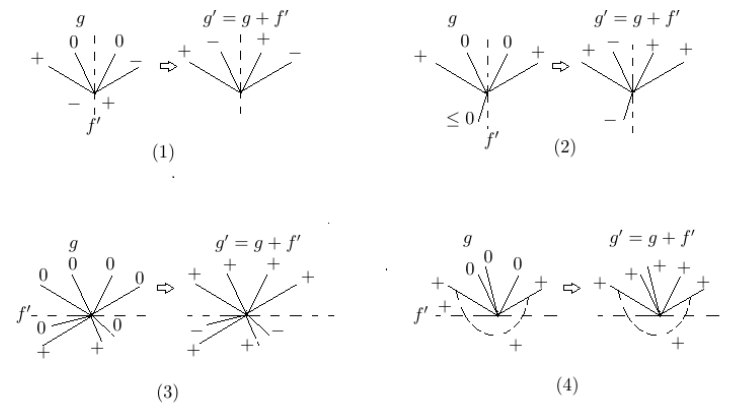}
  \caption{}
  \label{fig:1}
\end{figure}
 Second, we consider the case that $g(v_{k-1})>0$ and $g(v_{l+1})>0$ have the same sign. We will prove the case when both $g(v_{k-1})$ and $g(v_{l+1})$ are positive, and leave the completely analogous negative case to the reader. If there is some element $j\notin J\cup\{k-1,l+1\}$ such that $g(v_j)\leq0$,
 we can take a linear function $f'$ satisfying that $f'$ is negative on $v_j$ and on one or both of $v_1$ and $v_2$, $0\neq f'(v_i)$ for all $i\in I_0$ and $|f'(v_i)|<|g(v_i)|$ for all $i\in \{1,2,\ldots,n\}\backslash I_0$. Let $g'=g+f'$. We have $g'(v_{k-1})>0$, $g'(v_{l+1})>0$, $g'(v_j)<0$ and one or both of $g'(v_1)$ and $g'(v_2)$ is negative. Also we have $g'(v_i)\neq0$ for $i=1,\ldots,n$. Thus $I_{>0}=\{i|g'(v_i)>0\}$ is not connected. So $I_{>0} \in \Delta$. With the same argument as before, we obtain $\mathcal{L}$ is in the interior of $ C_{I_{>0}}$. See $(2)$ of Figure \ref{fig:1}.

It remains to consider the case when $g(v_i)>0$ for all $i \notin J\cup\{k-1,l+1\}$. Now $I_0$ is connected, we have $I_0=J$. Let $\theta$ be the angle formed by $v_k$ and $v_l$ and containing $v_1$. We know that $\theta\neq\pi$ by assumption. If $\theta > \pi$, we can take a linear function $f'$ satisfying that $f'(v_k)<0$ and $f'(v_l)<0$, $0\neq f'(v_i)$ for all $i\in I_0$ and $|f'(v_i)|<|g(v_i)|$ for all $i\in \{1,2,\ldots,n\}\backslash I_0$. There must exists an element $i_0\in J$ such that $f'(v_{i_0})>0$ since $\theta > \pi$. Let $g'=g+f'$, we have $g'(v_{k-1})>0$, $g'(v_k)<0$, $g'(v_{i_0})>0$, $g'(v_l)<0$, $g'(v_{l+1})>0$ and $g'(v_{i_0})<0$. Also we have $g'(v_i)\neq0$ for $i=1,\ldots,n$. Thus $I_{>0}=\{i|g'(v_i)>0\}$ is not connected. So $I_{>0} \in \Delta$. With the same argument as before, we obtain $\mathcal{L}$ is in the interior of $ C_{I_{>0}}$. See $(3)$ of Figure \ref{fig:1}. If $\theta < \pi$, we can take a linear function $f'$ satisfying $f'(v_i)>0$ for all $i\in I_0$ and $|f'(v_i)|<|g(v_i)|$ for all $i\in \{1,2,\ldots,n\}\backslash I_0$. Let $g'=g+f'$, we have $\{i|g'(v_i)>0\}=\{1,\ldots,n\} \in \Delta$. With the same argument as before, we obtain $\mathcal{L}$ in the interior of $ C_{\{1,\ldots,n\}}$. See $(4)$ of Figure \ref{fig:1}.

\end{proof}
\begin{definition}\label{h_I}
Let $V=\mathrm{Pic}_{\mathbb{R}}(\mathbb{P}_{\mathbf{\Sigma}})$. We pick a Euclidean metric on $V$.
Let $F_{I,1},\ldots,F_{I,k_I}$ be faces of cone $Z_I$ and $n_{I,i}$ be the unit normal vector of $F_{I,i}$ such that the inner product of $n_{I,i}$ with vector inside the cone is nonnegative. We define a linear function $h_{I,i}$ on $V$ by $$h_{I,i}(v)=\langle v, n_{I,i} \rangle$$ for $v\in V$, where $\langle v, n_{I,i} \rangle$ is inner product. Also, we define another two functions $h_I(v)$ and $\mathcal{E}(v)$ on $V$ by
\begin{equation*}
h_I(v)=\min\limits_{1\leq i\leq k_I}h_{I,i}(v) \text{ and } \mathcal{E}(v)=\max\limits_{ I\in \Delta}h_I(v)
\end{equation*}
 for $v\in V$. We know $\mathcal{E}(v)$ is a continuous function on $V$ since $h_{I,i}(v)$ is a continuous function on $V$ for each $I\in\Delta$ and $i=1,\ldots,k_I$.
\end{definition}

\begin{remark}\label{epsilon}
Let $S$ be the unit sphere in $V$. We know that the distance from a point $v\in Z_I$ to the face $F_{I,i}$ equals $h_{I,i}(v)$.
Also by Lemma \ref{interior}, we know that any $s\in S$
is contained in the interior of $Z_I$ for some $I\in \Delta$. So $\mathcal{E}(s)>0$ for all $s\in S$. Since $S$ is compact, there exists $\epsilon>0$ such that $\mathcal{E}(s)>\epsilon$ for all $s\in S$.
\end{remark}

We will show that arbitrary shifts of Forbidden cones cover almost all of $V$.
 We have $Z_I=\{v\in V | h_{I,i}(v)\geq 0 \text{ for } i=1,\ldots,k_I\}$. Moreover, for each $I\in \Delta$, we arbitrarily chose an element $x_I\in Z_I$.  Now we have $x_I+Z_I=\{v\in V|h_{I,i}(v)\geq h_{I,i}(x_I)\text{ for } i=1,\ldots,k_I\}$. Then we have the following lemma.

\begin{lemma}\label{d}
Let $x_I$ be arbitrary points in $Z_I$.
Assume that for any non-zero point $v\in V$, there exists some $I\in \Delta$ such that $v$ is in the interior of $Z_I$. Then there is a positive number $d$ such that if $|v|>d$, we have $v\in  x_I+Z_I$ for some $I\in\Delta$.
\end{lemma}
\begin{proof}
Pick $d$ such that $d>\frac{1}{\epsilon}h_{I,i}(x_I)$ for all $I$ and $1\leq i\leq k_I$. By Remark \ref{epsilon}, we have $\mathcal{E}(\frac{v}{|v|})>\epsilon$. By the definition of the function $\mathcal{E}$, there exists some $I\in \Delta$ such that $h_I(\frac{v}{|v|})>\epsilon$. Then $h_{I,i}(\frac{v}{|v|})>\epsilon$ for all $1\leq i\leq k_I$. Thus $h_{I,i}(v)>\epsilon |v|>\epsilon d> h_{I,i}(x_I)$ for all $1\leq i\leq k_I$. Then $h_{I,i}(v-x_I)>0$ for all $1\leq i\leq k_I$. This implies $v-x_I\in Z_I$.
\end{proof}

We apply Lemma \ref{d} to the shifts of the forbidden cones which have the property that their lattice points are in the forbidden sets. Such shifts exist by Proposition \ref{r_I}.
Since we have the assumption that $N=\sum_{i}^{n}\mathbb{Z}v_i$, then by Proposition \ref{r_I}, for each $I\in\Delta$, we take a point $r_I$ in $FS_{I}$ such that $r_I+(Z_I \cap \mathbb{Z}^k) \subseteq FS_{I}$, where $FS_{I}$ is the Forbidden set corresponding to $I$.
\begin{proposition}\label{inf->}
In the assumption of Theorem \ref{iff},
if no $v_i$ and $v_j$ are collinear, then $\mathbb{P}_{\mathbf{\Sigma}}$ has finitely many $\mathrm{H}-$trivial line bundles.
\end{proposition}
\begin{proof}
Assume $v_i$ and $v_j$ are not collinear for any $\{i, j\}\subset\{1,2,\cdots,n\}$. Then by Lemma \ref{interior}, any non-zero element $\mathcal{L}\in \mathrm{Pic}(\mathbb{P}_{\mathbf{\Sigma}}) \subset \mathrm{Pic}_{\mathbb{R}}(\mathbb{P}_{\mathbf{\Sigma}})$ is contained in the interior of $Z_I$ for some $I\in \Delta$. Then by Lemma \ref{d}, there is a positive number $d$ such that $|\mathcal{L}|>d$ implies $\mathcal{L}\in  r_I+Z_I$ for some $I\in \Delta$. Thus $\mathcal{L}$ is in $FS_I$ for some $I\in \Delta$. By Proposition \ref{Htrivial}, we obtain $\mathcal{L}$ is not $\mathrm{H}-$trivial. Thus all $\mathrm{H}-$trivial line bundles are contained in the ball in $\mathrm{Pic}_{\mathbb{R}} (\mathbb{P}_{\mathbf{\Sigma}})$ with radius $d$ and center at the origin. Because the torsion part of $\mathrm{Pic}(\mathbb{P}_{\mathbf{\Sigma}})$ is finite, this implies there are finitely many $\mathrm{H}-$trivial line bundles.
\end{proof}

\begin{proof}[Proof of Theorem \ref{iff}]
The result follows from Proposition \ref{->inf} and Proposition \ref{inf->}.
\end{proof}

\section{Further description of the infinite set of $\mathrm{H}-$trivial line bundles}\label{description}
In this section, when the set of $\mathrm{H}-$trivial line bundles is infinite, we show the set is of the form "finite set $+$ finite set of lines ". Thus the number of $\mathrm{H}-$trivial line bundles in a ball of radius $\alpha$
grows at most as constant times $\alpha$.

We want to show that Lemma \ref{interior} almost holds, in the sense that there are finitely many exceptions, up to scaling.
\begin{lemma}\label{notin}
 Let $\mathcal{L}$ be a point in $V=\mathrm{Pic}_{\mathbb{R}}(\mathbb{P}_{\mathbf{\Sigma}})$. Then $\mathcal{L}$ is not contained in the interior of $Z_I$ for any $I\in \Delta$ if and only if there exists collinear $v_p$ and $v_q$ such that $\mathcal{L}=\sum_{i\in J^+}h(v_i)E_i$ in $V$ for a linear function $h$ on $(\mathbb{R}^{2})^*$, where $h(v_p)=h(v_q)=0$ and $J^+=\{i\in\{1,\ldots,n\}| h(v_i)>0\}$.
\end{lemma}
\begin{proof}
Let $\mathcal{L}$ be an non-zero point in $V$ which is not contained in the interior of $Z_I$ for any $I\in \Delta$. Arbitrarily picking an element $j\in\{1,2\ldots,n\}$, by Lemma \ref{g}, we assume $\mathcal{L}=\sum_{i}g(v_i)E_i$, where $g$ is a $\mathbf{\Sigma}-$piece-wise linear function on $\mathbb{R}^2$ such that $g(v_j)=g(v_{j+1})=0$. Let $I_0=\{i|g(v_i)=0\}$. Since $\mathcal{L}$ is nonzero, we have $I_0\neq \{1,2,\ldots,n\}$. Let $J$ be the connected component of $I_0$ containing $v_j$ and $v_{j+1}$. We write the points in $J$ in clockwise order as $v_k,\ldots,v_j,v_{j+1},\ldots,v_l$. Applying the argument of Lemma \ref{interior}, we conclude that the angle $\theta$ between $v_k$ and $v_l$ must be $\pi$. Applying the same argument to $v_l$ and $v_{l+1}$, we get that $g$ is linear on $v_l, v_{l+1},\ldots, v_k$. If $g(v_l)>0$, we take $h=g$. If $g(v_l)<0$, we take $h=-g$. Then $\mathcal{L}=\sum_{i\in J^+}h(v_i)E_i$ in $V$, where $h(v_k)=h(v_l)=0$ and $J^+=\{i\in\{1,\ldots,n\}| h(v_i)>0\}$.

Now we consider the other direction. For any linear function $f\in (\mathbb{R}^2)^*$, let $\sum_{i}r_iE_i=\sum_{i\in I_{+}}h(v_i)E_i+\sum_{i=1}^{n}f(v_i)E_i$. We know  $\mathcal{L}$ is contained in the interior of $Z_I$ for some $I\in\Delta$ if and only if there exists a $f\in (\mathbb{R}^2)^*$ such that $r_i\neq 0$ for all $i\in\{1,\ldots,n\}$ and $I=\{i|r_i>0\}$. For any $f\in (\mathbb{R}^2)^*$, if $f(v_p)\neq 0$, we get that there are exactly two sign changes in $(r_i)_{i=1}^{n}$ using the same argument as the first case in the proof of Proposition \ref{->inf}. Thus if $r_i\neq 0$ for all $i\in\{1,\ldots,n\}$, we know $\{i\in\{1,\ldots,n\}|r_i>0\}$ is connected. If $f(v_p)=0$, we have $r_p=0$. Thus $\mathcal{L}$ is not contained in the interior of $Z_I$ for any $I\in\Delta$.
\end{proof}

For now, let us assume $v_p$ and $v_q$ are collinear and there are no other collinear pairs. We take a linear function $g$ such that $g(v_p)=g(v_q)=0$. Let $S$ be the unit sphere in $V$ with center origin, $P^+=\frac{\sum_{i\in J^+}g(v_i)E_i}{||\sum_{i\in J^+}g(v_i)E_i||}$ and $P^-=-P^+$. By Lemma \ref{notin}, we know that $P^+$ and $P^-$ are the only two points on $S$ that are not contained in the interior of $Z_I$ for any $I\in \Delta$.

For each $I\in \Delta$, Let $h_{I,i}$, $h_I$ and $\mathcal{E}$ be the functions defined in Definition \ref{h_I}.
If $P^+\in Z_I$, we know $P^+$ is on the boundary of $Z_I$. Let $\{F_{I,1},\ldots,F_{I,k}\}$ be all the faces of $Z_I$ that containing $P^+$.
Let $C^+$ be an open cone with center line $\mathbb{R}_{\geq 0}P^+$. Also, we can choose a sufficiently small $C^+$ which satisfies the following:
\begin{itemize}
\item $(1)$   $C^+ \cap \{Z_I| P^+\notin Z_I\}=\emptyset$, which implies if $P^+\notin Z_I$, we have $h_I(v)=0$ for any $v\in C^+$;
\item $(2)$   when $P^+\in Z_I$, we have $h_I(v)=\min\limits_{i\in\{1,\ldots,k\}}h_{I,i}(v)$.
\end{itemize}

Let $U$ be the orthogonal complement of $\mathbb{R}P^+$ in linear space $V$ and $$pr: V \rightarrow U$$ be the orthogonal projection. Let $\mathcal{E}$ be the function defined in Definition \ref{h_I}.
\begin{lemma}
There exists a positive constant $A$ such that $$\mathcal{E}(v)\geq A \, \mathrm{dist} (v, P^+)$$ for all points $v\in C^+$.
\end{lemma}
\begin{proof}
 We assume $Z_{I_1},\ldots, Z_{I_m}$ are all the cones containing $P^+$. By $(1)$, we have $\mathcal{E}(v)=\max\limits_{i\in\{1,\ldots,m\}}h_{I_i}(v)$ for $v\in C^+$. For each $i\in\{1,\ldots,m\}$, let $F_{I_i,1},\ldots,F_{I_i,k_i}$ be all the faces of $Z_{I_i}$ containing $P^+$. By $(2)$, we have $h_{I_i}(v)=\min\limits_{j\in\{1,\ldots,k_i\}}h_{I_i,j}(v)$. We can assume $v=aP^+ + pr(v)$ for some $a\in \mathbb{R}_{\geq0}$. Then we have $h_{I_i,j}(v)=h_{I_i,j}(pr(v))$.

 Define a function on $U$ by $$\mathcal{E}_U(u)=\max\limits_{I}\{\min\limits_{j\in \{1,\ldots,k_{I}\}}h_{I,j}(u)\}$$ for $u\in U$. For any $v\in C^+$, we have $\mathcal{E}(v)=\mathcal{E}_U(pr(v))$. We know $\mathcal{E}_U$ is continuous on $U$ and $\mathcal{E}_{U}(ku)=k\mathcal{E}_{U}(u)$ for $k>0$. By Lemma \ref{notin}, we know that $P^+$ and $P^-$ are the only two points on $S$ that are not contained in the interior of $Z_I$ for any $I$. So we have $\mathcal{E}$ takes positive value on $V \backslash \{P^+, P^-\}$. Then $\mathcal{E}_U$ takes positive value on $U\backslash\{0\}$ since $pr(v)\neq 0$ implies $v \notin \{P^+, P^-\}$. So there is a positive constance $A$ such that $\mathcal{E}_{U}(u)\geq A$ for any $u$ on the unit sphere in $U$. For any $u\in U\backslash\{0\}$, we have $\mathcal{E}_{U}(u)=|u|\mathcal{E}_{U}(\frac{u}{|u|})$. Thus $\mathcal{E}_{U}(u)\geq A|u|$ for $u \in U\backslash \{0\}$. This implies for any $v\in C^+ \backslash P^+$, we have $\mathcal{E}(v)\geq A |pr(v)|$. Since $|pr(v)|=\text{ dist} (v, P^+)$, we get the result.
\end{proof}

\begin{remark}
In the similar way, we have a positive constant $A'$ such that $\mathcal{E}(v)\geq A' \, \text{dist} (v, P^-)$ for all $v\in C^-$, where $C^-$ be an open cone with center line $\mathbb{R}_{\geq 0}P^-$ which is sufficiently small.
\end{remark}

\begin{proposition}\label{S}
Let $S$ be the unit sphere in $V$ and $W=S\backslash ((C^+\cap S)\sqcup (C^-\cap S))$.
There exist two positive constants $\lambda$ and $\epsilon$ such that $\mathcal{E}(v)\geq \epsilon$ for all $v \in W$ and $$\mathcal{E}(v)\geq \min \{\epsilon, \lambda \, \mathrm{dist} (v, P^+), \lambda \, \mathrm{dist} (v, P^-)\}$$ for all $v\in S$.
\end{proposition}
\begin{proof}
We have $S=(C^+\cap S)\sqcup (C^-\cap S)\sqcup W$. Since $(C^+\cap S)$ and $(C^-\cap S)$ are open in $S$, we get $W$ is closed. Thus $W$ is compact. So there is a positive constance $\epsilon$ such that $\mathcal{E}(v)\geq \epsilon$ for $v\in W$. Thus we have $\mathcal{E}(v)\geq \min \{\epsilon, A \, \text{dist} (v, P^+), A' \, \text{dist} (v, P^-)\}$ for all $v\in S \backslash \{P^+, P^-\}$. Let $\lambda=\min\{A,A'\}$, we get the conclusion.
\end{proof}
Let $a=\max\limits_{I, \{h_{I,i}\}_{i=1}^{k_I}}\{h_{I,i}(r_I)\}$. Let $B$ be a ball in $V$ with radium $\alpha$ and center origin. Also denote $T$ be a tube with center line $\mathbb{R}P^+$ and radium $\beta$. We have the following theorem.

\begin{theorem}\label{BT}
Let $\alpha>\frac{a}{\epsilon}$, $\beta>\frac{a}{\lambda}$ and $\mathcal{L}$ is an element in $\mathrm{Pic}(\mathbb{P}_{\mathbf{\Sigma}})$. Then we have $\mathcal{L}$ is not $H-$trivial if $\mathcal{L}$ is outside of $B \cup T$. See Figure \ref{fig:3}.
\end{theorem}
\begin{figure}[H]
  \includegraphics[width=0.4\textwidth]{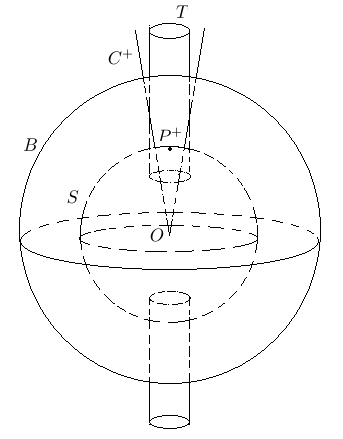}
  \caption{}
  \label{fig:3}
\end{figure}

\begin{proof}
Let $\mathcal{L}\notin B \cup T$, we claim $\mathcal{L}$ is $H-$trivial. Otherwise, $\mathcal{L}\notin r_I+Z_I$ for all $I$. Then there exists $h_{I,i}$ such that $h_{I,i}(\mathcal{L})<h_{I,i}(r_I)\leq a$. By the definition of $h_I$, we have $h_I(\mathcal{L})<a$ for all $I$. So we get $\mathcal{E}(\mathcal{L})<a$. If $\mathcal{L}\notin C^+ \cup C^-$, we have $\frac{\mathcal{L}}{|\mathcal{L}|} \in W$. Also, we know $\mathcal{E}(\frac{\mathcal{L}}{|\mathcal{L}|})<\frac{a}{|\mathcal{L}|}$. Since $\mathcal{L}\notin B$, we have $|\mathcal{L}|>\alpha>\frac{a}{\epsilon}$. So $\mathcal{E}(\frac{\mathcal{L}}{|\mathcal{L}|})<\epsilon$, which contradicts to Proposition \ref{S}. If $\mathcal{L}\in C^+$, we have $\mathcal{E}(\mathcal{L})\geq \lambda \, \text{dist} (v, P^+)$ by Proposition \ref{S}. Thus $\text{dist} (v, P^+)< \frac{a}{\lambda}<\beta$, which implies $\mathcal{L} \in T$. Thus we get contradiction.  The case that $\mathcal{L}\in C^-$ is similar.
\end{proof}

We give a description of $H-$trivial line bundles inside the tube by the following proposition.

\begin{proposition}\label{line}
Let $D_1$ and $D_2$ be two elements in $\mathrm{Pic}(\mathbb{P}_{\mathbf{\Sigma}})$. If $D_1+lD_2$ is $\mathrm{H}-$trivial for all $l\in\mathbb{Z}_{\geq0}$, then $D_1+lD_2$ is $\mathrm{H}-$trivial for all $l\in\mathbb{Z}$.
\end{proposition}
\begin{proof}
We know that $D_2$ does not lie in the interior of $Z_I$ for any $I\in \Delta$. The reason is that if $D_2$ lies in the interior of $Z_I$ for some $I\in \Delta$, we can find a sufficiently large $l\in\mathbb{Z}_{\geq0}$ such that $D_1+lD_2\in FS_I$. By Lemma \ref{notin}, we have $D_2=\sum_{i\in J^+}h(v_i)E_i$ in $\mathrm{Pic}_{\mathbb{R}}(\mathbb{P}_{\mathbf{\Sigma}})$ for a linear function $h$ on $(\mathbb{R}^{2})^*$, where $h(v_p)=h(v_q)=0$ and $J^+=\{i\in\{1,\ldots,n\}| h(v_i)>0\}$. That is to say there exists $f\in(\mathbb{R}^{2})^*$ such that $D_2+\sum_{i\in \{1,\ldots,n\}}f(v_i)E_i=\sum_{i\in J^+}h(v_i)E_i$. We can take a sufficiently large positive integer $m$ such that $mf(v_i)\in \mathbb{Z}$ and $m h(v_i)\in \mathbb{Z}$ for all $i\in \{1,\ldots,n\}$. Now we have $mD_2+\sum_{i\in \{1,\ldots,n\}}mf(v_i)E_i=\sum_{i\in J^+}mh(v_i)E_i$ and $mf\in N^*$. This implies $mD_2$ is equivalent to $\sum_{i\in J^+}mh(v_i)E_i$ by Proposition \ref{pic}. i.e. $mD_2=\sum_{i\in J^+}mh(v_i)E_i$ in $\mathrm{Pic}(\mathbb{P}_{\mathbf{\Sigma}})$. We denote $\sum_{i\in J^+}mh(v_i)E_i$ by $D_3$. Also we have $mh\in N^*$. Since $\sum_{i\in \{1,\ldots,n\} \backslash J^+}(-mh(v_i))E_i=D_3+\sum_{i\in \{1,\ldots,n\}}(-mh(v_i))E_i$, we have $D_3$ is equivalent to $\sum_{i\in \{1,\ldots,n\} \backslash J^+}(-mh(v_i))E_i$ which we denote by $D_4$. Note $-mh(v_i)>0$ for $i\in \{1,\ldots,n\} \backslash (J^+\cup \{p,q\})$. We get both $D_3$ and $D_4$ are effective. So we have the Koszul Complex
$$0\rightarrow \mathcal{O}(-D_3-D_4)\rightarrow \mathcal{O}(-D_3)\oplus\mathcal{O}(-D_4)\rightarrow \mathcal{O}\rightarrow \mathcal{O}_{D_3\cap D_4}\rightarrow 0.$$
Since the support of $D_3$ is $\cup_{i \in J^+} E_i$, the support of $D_4$ is $\cup_{i\in \{1,\ldots,n\} \backslash (J^+\cup \{p,q\})} E_i$ and the intersection of this two sets is empty, we have $\mathcal{O}_{D_3\cap D_4}$ is trivial. Thus we obtain $$0\rightarrow \mathcal{O}(-D_3-D_4)\rightarrow \mathcal{O}(-D_3)\oplus\mathcal{O}(-D_4)\rightarrow \mathcal{O}\rightarrow 0.$$ Also since both $D_3$ and $D_4$ are equivalent to $mD_2$, we get
\begin{equation}\label{exact}
0\rightarrow \mathcal{O}(-2mD_2)\rightarrow \mathcal{O}(-mD_2)\oplus\mathcal{O}(-mD_2)\rightarrow \mathcal{O}\rightarrow 0.
\end{equation}
Suppose there is a $l<0$ such that $D_1+lD_2$ is not $\mathrm{H}-$trivial. Let $l_0$ be the maximal one in the set $\{l\in \mathbb{Z}_{<0}|D_1+lD_2 \text{ is not } \mathrm{H}-\mathrm{trivial}\}$.
 We tensor the sequence \eqref{exact} by $\mathcal{O}(D_1+(l_0+2m)D_2)$ and get
$$0\rightarrow \mathcal{O}(D_1+l_0D_2)\rightarrow (\mathcal{O}(D_1+(l_0+m)D_2))^{\oplus 2}\rightarrow\mathcal{O}(D_1+(l_0+2m)D_2)\rightarrow 0.$$ Since the choice of $l_0$, we have both $\mathcal{O}(D_1+(l_0+m)D_2)$ and $\mathcal{O}(D_1+(l_0+2m)D_2)$ are $\mathrm{H}-$trivial. Thus $\mathcal{O}(D_1+l_0D_2)$ is $\mathrm{H}-$trivial, which leads to contradiction.
\end{proof}

\begin{remark}\label{line1}
All $\mathrm{H}-$trivial line bundles are within $B \cup T$. Let $O$ be the origin of $V$. Pick non-zero lattice point $Q \in T$. Let $H$ and $H'$ be two hyperplanes which are perpendicular to the vector $\overrightarrow{OQ}$ and pass through $Q$ and $O$ respectively. Then we get a cylinder by cutting the tube $T$ by $H$ and $H'$ which we denote by $Y$. We know every lattice point in $T$ equals a point in $Y$ plus $cQ$ for some $c\in \mathbb{Z}$. Since there are finitely many lattice points in $Y$, all the lattice points in $T$ are on finitely many parallel lines inside $T$.
\end{remark}
\begin{example}
Let $\Sigma$ be a complete simplicial fan in $N=\mathbb{Z}^2$, see $(1)$ of Figure \ref{fig:4}. Let $v_1=(1, 1), v_2=(0, 1), v_3=(-1, 0), v_4=(0, -1), v_5=(1, -1)$ to be the chosen lattice points in each of the one-dimensional cones of $\Sigma$. We consider the toric DM stack $\mathbb{P}_{\mathbf{\Sigma}}$ associated to this stacky fan $\mathbf{\Sigma}=(\Sigma, \{v_i\}_{i=1}^{n})$. By Theorem \ref{iff}, there are infinitely many $\mathrm{H}-$trivial line bundles on $\mathbb{P}_{\mathbf{\Sigma}}$.
\begin{figure}[H]
  \includegraphics[width=1.1\textwidth]{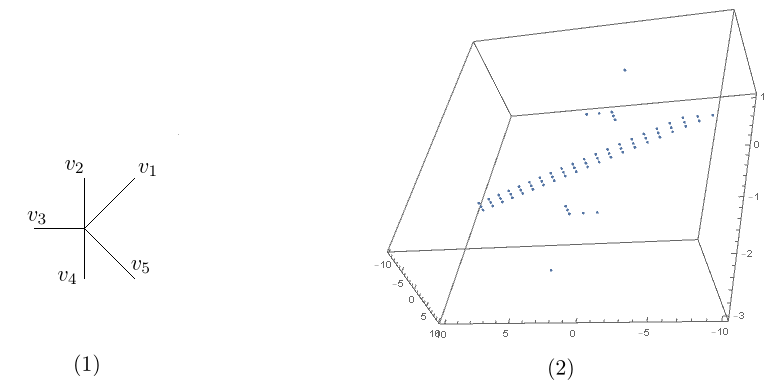}
  \caption{}
  \label{fig:4}
\end{figure}
The Picard group $\mathrm{Pic}(\mathbb{P}_{\mathbf{\Sigma}})$ is generated by $E_1, E_4$ and $E_5$. A point $aE_1+bE_4+cE_5$ in $\mathrm{Pic}(\mathbb{P}_{\mathbf{\Sigma}})$ is denoted by $(a,b,c)$.
We computed $\mathrm{H}-$trivial line bundles to be $\tau_1=\{(a,-1,a)|a\in \mathbb{Z}\}$, $\tau_2=\{(a,-1,a-1)|a\in \mathbb{Z}\}$ and $\tau_3=\{(a,-1,a-2)|a\in \mathbb{Z}\}$  which lie in three lines respectively, and a finite set as follows:
\begin{equation*}
\begin{split}
&(1,0,-1), (-1,0,-1), (-1,0,0), (-1,0,1), (-1,1,0), (0,0,-1),\\
&(-2,-2,-2),(-1,-2,-2),(0,-2,-2),(0,-3,-3),(0,-2,-3),(0,-2,-4)
\end{split}
\end{equation*}
See $(2)$ of Figure \ref{fig:4}. This is consistent with Theorem \ref{BT}, Proposition \ref{line} and Remark \ref{line1}.
\end{example}
If we have more than one pair of collinear vectors in $\{v_1,\ldots,v_n\}$, then the following theorem holds.
\begin{theorem}\label{BT1}
Let $\{v_{p_i},v_{q_i}\}_{i=1}^{l}$ be all collinear pairs in $\mathbf{\Sigma}$. There are $l$ tubes $\{T_1,\ldots, T_l\}$ with center line passing through $O$ and sufficiently large radius, and a ball $B$ with center $O$ and sufficiently large radium such that any line bundle outside of $B\bigcup (\cup_{i=1}^{l} T_i)$ is not $\mathrm{H}-$trivial.
\end{theorem}
\begin{proof}
The argument of Theorem \ref{BT} is repeated for each collinear pair. Details are left to the reader.
\end{proof}
\section{Generalization to $\mathrm{dim}\, \mathrm{H}^0+ \mathrm{dim}\, \mathrm{H}^1+ \mathrm{dim}\, \mathrm{H}^2<m$}\label{gene}
This section contain a generalization of the main theorem.

Let $\mathrm{Pic} (\mathbb{P}_{\mathbf{\Sigma}})$ be the Picard group of $\mathbb{P}_{\mathbf{\Sigma}}$. For any positive integer $m$, we consider the set $$\Lambda_m=\{\mathcal{L}\in \mathrm{Pic}( \mathbb{P}_{\mathbf{\Sigma}}) |  \mathrm{dim}\, \mathrm{H}^0(\mathbb{P}_{\mathbf{\Sigma}},\mathcal{L})+ \mathrm{dim}\, \mathrm{H}^1(\mathbb{P}_{\mathbf{\Sigma}},\mathcal{L})+ \mathrm{dim}\, \mathrm{H}^2(\mathbb{P}_{\mathbf{\Sigma}},\mathcal{L})<m\}.$$ We know $\Lambda_1=\{\mathcal{L}\in \mathrm{Pic}( \mathbb{P}_{\mathbf{\Sigma}}) | \mathcal{L} \text{ is $\mathrm{H}-$trivial} \}$. Also by Theorem \ref{iff}, we know $\Lambda_1$ is finite iff there exists no $\{i, j\}\subset\{1,2,\ldots,n\}$ such that $v_i$ and $v_j$ are collinear. We have the following theorem.

\begin{theorem}\label{generalize}
In the assumption of Theorem \ref{iff},
if no $v_i$ and $v_j$ are collinear, then $\Lambda_m$ is finite for any positive integer $m$.
\end{theorem}
\begin{proof}
By Lemma \ref{ktimes}, for each $I$, we have an element $\eta_I\in Z_I$ such that any element $\mathcal{L} \in (\eta_I+Z_I)\cap \mathrm{Pic}( \mathbb{P}_{\mathbf{\Sigma}})$ can be expressed as a linear combination
of generators of $Z_I$ in at least $m$ ways. And each expression contributes at least one to $\mathrm{dim}\, \mathrm{H}^0(\mathbb{P}_{\mathbf{\Sigma}},\mathcal{L})+ \mathrm{dim}\, \mathrm{H}^1(\mathbb{P}_{\mathbf{\Sigma}},\mathcal{L})+ \mathrm{dim}\, \mathrm{H}^2(\mathbb{P}_{\mathbf{\Sigma}},\mathcal{L})$ by Proposition \ref{Coho}. Thus $\mathcal{L} \notin \Lambda_m$. Also by Proposition \ref{inf->}, we have that any non-zero element in $\mathrm{Pic}_{\mathbb{R}}(\mathbb{P}_{\mathbf{\Sigma}})$ is contained in the interior of $Z_I$ for some $I$ since $v_i$ and $v_j$ are not collinear for any $\{i, j\}\subset\{1,2,\cdots,n\}$. Then by Lemma \ref{d},
 there is a positive number $d'$ such that if $|L|>d'$, we have $\mathcal{L}\in  \eta_I+Z_I$ for some $I$. Thus $\Lambda_m \subseteq \{L \in \mathrm{Pic} (\mathbb{P}_{\mathbf{\Sigma}})| |\mathcal{L}|\leq d' \}$, which implies $\Lambda_m$ is finite.
\end{proof}
\begin{remark}
The tubes $+$ ball description of $H-$trivial line bundles also applies to $\Lambda_m$. The argument is the same.
\end{remark}

\section{Comments}\label{comments}
In this section, we express our expectation for the case of dimension three.

We consider the smooth toric DM stack $\mathbb{P}_{\mathbf{\Sigma}}$ associated to a complete stacky fan $\mathbf{\Sigma}=(\Sigma, \{v_i\}_{i=1}^{n})$ in a lattice with rank $3$. Similarly to dimension two case, infinitely many $\mathrm{H}-$trivial line bundles can be obtained by having a fibration $\pi$ from $\mathbb{P}_{\mathbf{\Sigma}}$ to a certain base and a line bundle $\mathcal{L}$ on $\mathbb{P}_{\mathbf{\Sigma}}$ such that the higher direct image $\mathbf{R}^i \pi_*(\mathcal{L})=0$ for all $i\geq 0$. The following conjecture is meant to encode the existence of such fibrations with two- or one-dimensional base.

\begin{conjecture}\label{iffd3}
There are infinitely many $\mathrm{H}-$trivial line bundles on $\mathbb{P}_{\mathbf{\Sigma}}$ if and only if there exists $\{i, j\}\subset\{1,2,\cdots,n\}$ such that $v_i$ and $v_j$ are collinear or there exists a plane intersecting all three dimensional cones of $\mathbf{\Sigma}$ at their boundaries.
\end{conjecture}

It is hoped that the methods and approach of this paper will still be useful for Conjecture \ref{iffd3}.
One of the reasons to study $\mathrm{H}-$trivial line bundles on dim $3$ smooth DM Fano stacks is to try to find an example of one without full exceptional collection of line bundles. Conjecture \ref{iffd3} indicates that such example is likely to have origin far from all diagonals.

\section{Appendix: several facts about cones and semigroups}\label{appendix}

In this section, we state and prove several facts about the relationship between lattice points in a cone and semigroup generated by several lattice points in a finitely generated group $M$. We assume $M=\mathbb{Z}^k\oplus M_{tor}$ for some positive integer $k$, where $M_{tor}$ is torsion part of $M$.

Let $w_1,\ldots,w_n$ be $n$ elements of $M$ such that $M=\sum_{i=1}^{n}\mathbb{Z}w_i$, that is to say $w_1,\ldots,w_n$ generate $M$ over $\mathbb{Z}$. We assume $C=\sum_{i=1}^{n}\mathbb{R}_{\geq0}w_i$ is a cone satisfying $C\cap (-C)=\{0\}$. Then we pick a supporting hyperplane such that all $w_i$ are on the same side of the plane, that is to say the linear function corresponding to the hyperplane which we denote by $h$ takes positive value on all $w_i$.

\begin{lemma}\label{<}
Let $x$ be a point in $C$. If $h(x)\geq \sum_{i=1}^{n}h(w_i)$, there is a point $y\in C$ such that $x=y+w_j$ for some $j\in \{1,\ldots,n\}$.
\end{lemma}
\begin{proof}
We assume $x=\sum_{i=1}^{n}a_iw_i$ $(\mathrm{mod} M_{tor})$, where $a_i\in \mathbb{R}_{\geq 0}$. Then we have $h(x)=\sum_{i=1}^{n}a_ih(w_i)$. If $0\leq a_i<1$ for $i=1,\ldots,n$, we get $h(x)<\sum_{i=1}^{n}h(w_i)$ which contradicts our assumption. Thus there is $a_j>1$ for some $j\in \{1,\ldots,n\}$. Let $y=x-w_j$. So $y$ is a linear combination of $w_i$ with nonnegative coefficients and $x=y+w_j$.
\end{proof}
\begin{corollary}
Let $FS=\sum_{i=1}^{n}\mathbb{Z}_{\geq0}w_i$ be the semigroup generated by $w_i$. For any point $x\in C$, we can write $x=a+b$, where $a\in FS$ and $b\in C$ with $h(b)<\sum_{i=1}^{n}h(w_i)$.
\end{corollary}
\begin{proof}
If $h(x)<\sum_{i=1}^{n}h(w_i)$, we have $x=a+b$, where $a=0$ and $b=x$. If $h(x)\geq \sum_{i=1}^{n}h(w_i)$, by Lemma \ref{<}, there exits some $j_1\in \{1,\ldots,n\}$ such that $x-w_j\in C$. If $h(x-w_{j_1})<\sum_{i=1}^{n}h(w_i)$, we have $x=a+b$, where $a=w_j$ and $b=x-w_j$. If $h(x-w_{j_1})\geq\sum_{i=1}^{n}h(w_i)$, we use Lemma \ref{<} again to get some $j_2\in\{1,\ldots,n\}$ such that $x-w_{j_1}-w_{j_2}\in C$. We repeat the process which will stop within finite steps. That is to say, there is an integer $m$ such that all the coefficient of $x-\sum_{l=1}^{m}w_{j_l} \in C$ are less than one. Thus $h(x-\sum_{l=1}^{m}w_{j_l})<\sum_{i=1}^{n}h(w_i)$. So $x=a+b$, where $a=\sum_{l=1}^{m}w_{j_l}\in FS$ and $b=x-\sum_{l=1}^{m}w_{j_l}$.
\end{proof}
\begin{remark}
We consider a set $\{p\in C \cap M| h(p)< \sum_{i=1}^{n}h(w_i)\}$ which we denote by $\Gamma$. Since the lattice points in a bounded region are finite, we know $\Gamma$ is a finite set. Then we have the following corollary.
\end{remark}
\begin{corollary}\label{cup}
We assume $\Gamma=\{p_1,\ldots,p_t\}$. Then we have
$$C \cap M=\bigcup_{i=1}^{t}(p_i+FS).$$
\end{corollary}
\begin{proposition}\label{r_I}
Let $FS=\sum_{i=1}^{n}\mathbb{Z}_{\geq0}w_i$. Then there exists a point $r\in FS$ such that $r+(C \cap M) \subseteq FS$.
\end{proposition}
\begin{proof}
Since $M=\sum_{i=1}^{n}\mathbb{Z}w_i$, we assume $p_i=\sum_{j}a_{ij}w_j$ $(\mathrm{mod} M_{tor})$, where $a_{ij}\in  \mathbb{Z}$. Let $a_j=\max\limits_{1\leq i\leq t}{|a_{ij}|}$ for $j=1,\ldots,n$. Denote $r=\sum_{j=1}^{n}a_jw_j$. We claim $r+(C \cap M) \subseteq FS$. For any point $p\in C \cap M $, $p\in p_i+FS$ for some $i\in\{1,\ldots,t\}$ by Corollary \ref{cup}. We assume $p=p_i+s$, where $s\in FS$. We have all the coefficients of $r+p_i$ are in $\mathbb{Z}_{\geq 0}$ by the definition of $a_j$. Thus $r+p_i\in FS$. So $r+p=r+p_i+s\in FS$.
\end{proof}
We also have a lemma.
\begin{lemma}\label{ktimes}
If $n>k$, there exists a point $p\in FS$ such that every element in $r+p+(C\cap M)$ can be written as a nonnegative integer linear combination of $w_1,\ldots,w_n$ in at least $m$ different ways.
\end{lemma}
\begin{proof}
Since $n>k$ and $M=\sum_{i=1}^{n}\mathbb{Z}w_i$, there exists a relation $\sum_{i=1}^{n}a_iw_i=0$, where $a_i\in \mathbb{Z}$ and $\{i|a_i\neq0\}\neq\emptyset$. Let $I_+=\{i|a_i>0\}$ and $I_-=\{i|a_i<0\}$. We claim $I_+\neq \emptyset$. Otherwise, we have $\sum_{i=1}^{n}a_ih(w_i)<0$ since the assumption that $h$ takes positive value on all $w_i$. This contradicts to $\sum_{i=1}^{n}a_ih(w_i)=h(\sum_{i=1}^{n}a_iw_i)=h(0)=0$. Similarly, we have $I_-\neq \emptyset$. Then let $p_1=\sum_{i\in J_+}a_iw_i=\sum_{i\in J_-}(-a_i)w_i$. Let $p=(m+1)p_1$. We have
\begin{equation*}
\begin{split}
p&=mp_1+p_1=\sum_{i\in J_+}ma_iw_i+\sum_{i\in J_-}(-a_i)w_i\\
&=(m-1)p_1+2p_1=\sum_{i\in J_+}(m-1)a_iw_i+\sum_{i\in J_-}2(-a_i)w_i\\
&=\cdots\\
&=p_1+mp_1=\sum_{i\in J_+}a_iw_i+\sum_{i\in J_-}m(-a_i)w_i
\end{split}
\end{equation*}
Thus $p$ can be written as a nonnegative integer linear combination of $w_1,\ldots,w_n$ in at least $m$ different ways.
Then for arbitrary $x\in r+(C\cap M)$, we assume $x=\sum_{i=1}^{n}b_iw_i$ for $b_i\in \mathbb{Z}_\geq 0$ by Lemma \ref{r_I}. We have $x+p=\sum_{i=1}^{n}b_iw_i+\sum_{i\in J_+}(m+1-j)a_iw_i+\sum_{i\in J_-}j(-a_i)w_i$ for $j=1,\ldots,m$. So $x+p$ can be written as a nonnegative integer linear combination of $w_1,\ldots,w_n$ in at least $m$ different ways.
\end{proof}

\end{document}